\documentclass[11pt]{article}
\usepackage{amsthm, amsmath, amssymb, amsfonts, url, booktabs, tikz, setspace, fancyhdr, bm}
\usepackage{hyperref}
\usepackage{geometry}
\usepackage{hyperref, enumerate}
\usepackage[shortlabels]{enumitem}
\usepackage[babel]{microtype}
\usepackage[english]{babel}
\usepackage[capitalise]{cleveref}
\hypersetup{hidelinks}
\usepackage{comment}
\usepackage{bbm}
\usepackage{csquotes}
\usepackage{mathabx}
\usepackage{tikz}
\usepackage{graphicx}
\usepackage{float}
\usepackage{amsmath}

\counterwithin{figure}{section}

\newtheorem{theorem}{Theorem}[section]
\newtheorem{prop}[theorem]{Proposition}
\newtheorem{lemma}[theorem]{Lemma}
\newtheorem{cor}[theorem]{Corollary}

\newtheorem{claim}[theorem]{Claim}

\theoremstyle{definition}
\newtheorem{defn}[theorem]{Definition}
\newtheorem*{defn-non}{Definition}

\newlist{Case}{enumerate}{3}
\setlist[Case, 1]{%
    label           =   {\bfseries Case \arabic*.},
    labelindent=1em ,labelwidth=1cm, labelsep*=1em, leftmargin =!
}
\setlist[Case, 2]{%
    label           =   {\bfseries Subcase \arabic{Casei}.\arabic*.},
    labelindent=-1em ,labelwidth=1cm, labelsep*=1em, leftmargin =!
}
\setlist[Case, 3]{%
    label           =   {\bfseries Subsubcase \arabic{Casei}.\arabic{Caseii}.\arabic*.},
    labelindent=-1em ,labelwidth=1cm, labelsep*=1em, leftmargin =!
}

\newenvironment{poc}{\begin{proof}[Proof of claim]}{\end{proof}}

\usepackage{todonotes}

\newcommand{\cP}{\mathcal{P}}

\newcommand{\VC}{\operatorname{VC}}
\newcommand{\Tr}{\operatorname{Tr}}
\newcommand{\cF}{\mathcal{F}}

\newcommand{\cA}{\mathcal{A}}
\newcommand{\cB}{\mathcal{B}}
\newcommand{\cC}{\mathcal{C}}
\newcommand{\cH}{\mathcal{H}}
\newcommand{\cK}{\mathcal{K}}
\newcommand{\cR}{\mathcal{R}}
\newcommand{\cD}{\mathcal{D}}

\title{Recursive Lifting Beyond the Ahlswede--Khachatrian Construction}

\author{Xiaochen Zhao\thanks{School of Mathematical Sciences, Capital Normal University, Beijing, China. Email: 2250501013@cnu.edu.cn.}
\and
Gennian Ge\thanks{School of Mathematical Sciences, Capital Normal University, Beijing, China. Email: gnge@zju.edu.cn. Gennian Ge was supported by the National Key Research and Development Program of China under Grant 2025YFC3409900, the National Natural Science Foundation of China under Grant 12231014, and Beijing Scholars Program.}}

\date{}

\begin{document}
\maketitle
\begin{abstract}
For the Erd\H{o}s--Frankl--Pach problem on uniform set systems of bounded VC-dimension, the Ahlswede--Khachatrian/Mubayi--Zhao construction has long served as the standard lower-bound benchmark. We develop a recursive lifting method that goes beyond this benchmark in every dimension \(d\ge3\), proving that for every \(d\ge3\) and \(n\ge d+3\),
\[
M_d(n)\ge
\binom{n-1}{d}+\binom{n-4}{d-2}+M_{d-3}(n-5).
\]
The proof is elementary and proceeds through explicit trace obstructions. 
We also record a further recursive improvement in the concluding remarks.


\end{abstract}

\section{Introduction}
Let $X$ be a finite set.  For a family $\cF\subseteq2^X$ and a set $S\subseteq X$, the trace of $\cF$ on $S$ is
\[
        \Tr_{\cF}(S)=\{F\cap S:F\in\cF\}.
\]
The set $S$ is shattered by $\cF$ if $\Tr_{\cF}(S)=2^S$, and the VC-dimension $\VC(\cF)$ is the maximum size of a shattered set.  The Sauer--Shelah lemma states that every family $\cF\subseteq2^{[n]}$ with $\VC(\cF)\le d$ has size at most $\sum_{i=0}^d\binom ni$, and this is sharp, for instance by taking all sets of size at most $d$; see \cite{Sauer1972,Shelah1972,VC1971}.

The uniform analogue is much less rigid.  For $d\ge0$ and $n\ge d+1$, define
\[
 M_d(n)=\max\left\{|\cF|:\cF\subseteq\binom{[n]}{d+1},\ \VC(\cF)\le d\right\}.
\]
Erd\H{o}s \cite{Erdos1984} and, independently, Frankl and Pach \cite{FranklPach1984} initiated the uniform version of this problem.  Frankl and Pach proved the fundamental upper bound
\[
        M_d(n)\le \binom nd.
\]
They further conjectured that for sufficiently large $n$ the extremal family should be the star, of size $\binom{n-1}{d}$.  This was disproved by Ahlswede and Khachatrian \cite{AhlswedeKhachatrian1997}, who constructed a family of size
\begin{equation}\label{eq:AK}
        \binom{n-1}{d}+\binom{n-4}{d-2}
\end{equation}
whenever $n\ge2d+2$.  Mubayi and Zhao \cite{MubayiZhao2007} later found many non-isomorphic constructions with the same size and conjectured that \eqref{eq:AK} is best possible for all sufficiently large $n$.

Recent progress on the upper-bound side has shown that the Frankl--Pach bound is far from the final answer: Ge, Xu, Yip, Zhang and Zhao \cite{GeXuYipZhangZhao2026} proved that $\binom nd$ is never attained for every \(d\ge2\) and \(n\ge2d+2\), and Chao, Xu, Yip and Zhang \cite{ChaoXuYipZhang2025} and Yang and Yu \cite{YangYu2025} obtained asymptotic improvements for fixed $d$.  On the lower-bound side, however, the Ahlswede--Khachatrian/Mubayi--Zhao construction remained the basic general benchmark.  The case $d=2$ even supports this picture: Wang, Xu and Zhang \cite{WangXuZhang2025} proved that $M_2(n)=\binom{n-1}{2}+1$ for $n\ge7$.

The purpose of this paper is to show that the lower-bound picture changes in every dimension $d\ge3$.  Our construction is recursive.  It starts from a lifted covering-pair formulation of the Ahlswede--Khachatrian construction and enlarges the overlap by inserting an arbitrary lower-dimensional family.  This gives the following bound.

\begin{theorem}\label{thm:main}
For every $d\ge3$ and $n\ge d+3$,
\begin{equation}\label{eq:main}
        M_d(n)\ge \binom{n-1}{d}+\binom{n-4}{d-2}+M_{d-3}(n-5).
\end{equation}
\end{theorem}

Since a star gives $M_{d-3}(n-5)\ge\binom{n-6}{d-3}$, we immediately obtain a clean explicit consequence.

\begin{cor}\label{cor:explicit}
For every $d\ge3$ and $n\ge d+3$,
\[
        M_d(n)\ge \binom{n-1}{d}+\binom{n-4}{d-2}+\binom{n-6}{d-3}.
\]
In particular, in the usual range \(n\ge 2d+2\) where the Ahlswede--Khachatrian/Mubayi--Zhao construction gives \eqref{eq:AK}, this value is not optimal for any \(d\ge3\).
\end{cor}

In an earlier preprint~\cite{ZhaoGe2026v1}, the authors gave a recursive lower-bound construction beyond the Ahlswede--Khachatrian/Mubayi--Zhao benchmark.
Almost simultaneously, Tran and Xu~\cite{TranXu2026} obtained an independent construction beating the same bound \eqref{eq:AK}.  Motivated in part by this parallel construction, the present paper records a further recursive augmentation in the concluding remarks.

Throughout the paper we use the convention that $\binom Xr=\emptyset$ if $r<0$ or $r>|X|$, and $\binom mr=0$ for numerical binomial coefficients outside this range. 

\section{Recursive lifting}

We first record the elementary uniform criterion which allows us to verify VC-dimension by checking only members of the family.

\begin{lemma}\label{lem:uniform}
Let $\cF\subseteq\binom Xr$.  Then $\VC(\cF)\le r-1$ if and only if no member of $\cF$ is shattered by $\cF$.
\end{lemma}

\begin{proof}
If some $F\in\cF$ is shattered, then $\VC(\cF)\ge |F|=r$.  Conversely, if $\cF$ shatters a set of size at least $r$, then by heredity it shatters an $r$-set $S$.  The full trace $S$ is realized by some $F\in\cF$ with $S\subseteq F$; since $|S|=|F|=r$, this forces $F=S$.  Thus a member of $\cF$ is shattered.
\end{proof}

The next definition isolates the lifting mechanism.

\begin{defn}\label{def:covering}
Let $W$ be finite and let $\cA,\cB\subseteq\binom Wd$.  We call $(\cA,\cB)$ an admissible covering pair if
\[
        \cA\cup\cB=\binom Wd
\]
and every $S\in\cA\cap\cB$ is shattered by neither $\cA$ nor $\cB$.
\end{defn}

Given two new points $a,b\notin W$, define the lift
\begin{equation}\label{eq:lift}
 L(\cA,\cB)=
 \{\{a,b\}\cup R:R\in\binom W{d-1}\}
 \cup\{\{a\}\cup A:A\in\cA\}
 \cup\{\{b\}\cup B:B\in\cB\}.
\end{equation}

\begin{lemma}[Two-cover lifting]\label{lem:lifting}
If $(\cA,\cB)$ is an admissible covering pair on an $m$-element set $W$, then
\[
        L(\cA,\cB)\subseteq\binom{W\cup\{a,b\}}{d+1},\qquad \VC(L(\cA,\cB))\le d,
\]
and
\[
        |L(\cA,\cB)|=\binom{m+1}{d}+|\cA\cap\cB|.
\]
\end{lemma}

\begin{proof}
Put $\cF=L(\cA,\cB)$.  It is enough, by Lemma \ref{lem:uniform}, to show that no member of $\cF$ is shattered.  If $E=\{a,b\}\cup R$ with $R\in\binom W{d-1}$, then the trace $R$ is missing on $E$, since every member of $\cF$ contains at least one of $a,b$.  If $E=\{a\}\cup S$ with $S\in\cA$, then a trace on $E$ not containing $a$ can only come from a set $\{b\}\cup B$ with $B\in\cB$.  If $S\notin\cB$, then the full trace $S$ is absent from $\Tr_{\cB}(S)$; if $S\in\cA\cap\cB$, then $S$ is not shattered by $\cB$ by admissibility.  Thus some trace not containing $a$ is missing.  The case $E=\{b\}\cup S$ with $S\in\cB$ is symmetric.  Finally,
\[
        |\cF|=\binom m{d-1}+|\cA|+|\cB|
        =\binom m{d-1}+\binom md+|\cA\cap\cB|
        =\binom{m+1}{d}+|\cA\cap\cB|.
\]
\end{proof}

We now construct an admissible pair with a recursively enlarged overlap.  Let
\[
        W=\{\alpha,\beta,\gamma\}\cup V,
\]
and let $\cH\subseteq\binom V{d-2}$ satisfy $\VC(\cH)\le d-3$.  Define
\begin{align}
\cA_0={}&\{S\in\binom Wd:\alpha\notin S,\ \beta\in S\}
       \cup\{S\in\binom Wd:S\cap\{\alpha,\beta,\gamma\}=\emptyset\}, \label{eq:A0}\\
\cB_0={}&\{S\in\binom Wd:\alpha\in S\}
       \cup\{S\in\binom Wd:\alpha,\beta\notin S,\ \gamma\in S\}.       \label{eq:B0}
\end{align}
These two families partition $\binom Wd$.  Put
\begin{align}
\cC_0&=\{\{\alpha,\beta\}\cup R:R\in\binom{W\setminus\{\alpha,\beta\}}{d-2}\},\label{eq:C0}\\
\cC_1&=\{\{\alpha,\gamma\}\cup Q:Q\in\cH\},\label{eq:C1}\\
\cD&=\{\{\beta,\gamma\}\cup Q:Q\in\cH\},\label{eq:D}
\end{align}
and set
\begin{equation}\label{eq:AB}
        \cA=(\cA_0\setminus\cD)\cup\cC_0\cup\cC_1,
        \qquad
        \cB=\cB_0\cup\cD.
\end{equation}

\begin{lemma}\label{lem:overlap}
The pair $(\cA,\cB)$ defined in \eqref{eq:A0}--\eqref{eq:AB} is admissible, and
\[
        |\cA\cap\cB|=\binom{|W|-2}{d-2}+|\cH|.
\]
\end{lemma}

\begin{proof}
Since $\cA_0$ and $\cB_0$ partition $\binom Wd$, while $\cC_0\cup\cC_1\subseteq\cB_0$ and $\cD\subseteq\cA_0$, we have $\cA\cup\cB=\binom Wd$ and $\cA\cap\cB=\cC_0\cup\cC_1$.  The count follows.  It remains to check that no member of the overlap is shattered by either side.

Let $S=\{\alpha,\beta\}\cup R\in\cC_0$.  If $\gamma\in R$, then $\{\alpha\}\notin\Tr_{\cA}(S)$: members of $\cA_0$ avoid $\alpha$, members of $\cC_0$ contain $\beta$, and members of $\cC_1$ contain $\gamma$.  Similarly $\{\beta\}\notin\Tr_{\cB}(S)$, because members of $\cB_0$ either contain $\alpha$ or, if they avoid both $\alpha,\beta$, contain $\gamma$, while members of $\cD$ contain both $\beta$ and $\gamma$.  If $\gamma\notin R$, then $R\subseteq V$ and $|R|=d-2$.  Since $\VC(\cH)\le d-3$, choose $J\subseteq R$ with $J\notin\Tr_{\cH}(R)$.  Then $\{\alpha\}\cup J\notin\Tr_{\cA}(S)$ and $\{\beta\}\cup J\notin\Tr_{\cB}(S)$: the only possible realizers with the right special point would lie in $\cC_1$ and $\cD$, respectively, and would require a member of $\cH$ with trace $J$ on $R$.

It remains to consider $S=\{\alpha,\gamma\}\cup Q\in\cC_1$, where $Q\in\cH$.  We have $\{\gamma\}\cup Q\notin\Tr_{\cA}(S)$.  Indeed, a member of $\cA_0$ realizing this trace would have to be $\{\beta,\gamma\}\cup Q$, which was deleted into $\cD$, while members of the second part of $\cA_0$ avoid $\gamma$ and members of $\cC_0\cup\cC_1$ contain $\alpha$.  Also $\emptyset\notin\Tr_{\cB}(S)$, since every member of $\cB_0\cup\cD$ contains at least one of $\alpha,\gamma$.  Thus every overlap member is shattered by neither family.
\end{proof}

\begin{proof}[Proof of Theorem \ref{thm:main}]
Let $V$ be a set of size $n-5$, and choose
\[
        \cH\subseteq\binom V{d-2},\qquad \VC(\cH)\le d-3,
        \qquad |\cH|=M_{d-3}(n-5).
\]
Set $W=\{\alpha,\beta,\gamma\}\cup V$, so $|W|=n-2$.  By Lemma \ref{lem:overlap}, there is an admissible covering pair $(\cA,\cB)$ on $W$ with
\[
        |\cA\cap\cB|=\binom{n-4}{d-2}+M_{d-3}(n-5).
\]
Adding two new points $a,b$ and applying Lemma \ref{lem:lifting} gives a $(d+1)$-uniform family on $n$ points with VC-dimension at most $d$ and size
\[
        \binom{n-1}{d}+\binom{n-4}{d-2}+M_{d-3}(n-5),
\]
which proves the theorem.
\end{proof}

\begin{proof}[Proof of Corollary \ref{cor:explicit}]
Fix a point $v$ in an $(n-5)$-element set $V$ and take the star
\[
        \cH=\{H\in\binom V{d-2}:v\in H\}.
\]
It is intersecting, so the empty trace is missing on every member of $\cH$; by Lemma \ref{lem:uniform}, $\VC(\cH)\le d-3$. Since $|\cH|=\binom{n-6}{d-3}$, substitution in Theorem \ref{thm:main} gives the result.
\end{proof}

\section{Concluding remarks}\label{sec:concluding}
For the augmentation below, fix \(d\ge4\) and \(n\ge d+3\).  Let \(U\) be a set of size \(n-6\), let \(x\notin U\), and put
\(V=\{x\}\cup U,\text{ }W=\{\alpha,\beta,\gamma\}\cup V\).
We keep the notation from the proof of Theorem \ref{thm:main}; thus \(\cA_0,\cB_0,\cC_0\) are defined by \eqref{eq:A0}--\eqref{eq:C0} with this choice of \(W\).  Let \(a,b\) be the two lifting points, and use concatenation to denote union of labelled points. In the present notation, the Ahlswede--Khachatrian contribution is
\[
\begin{aligned}
\mathcal L_{\mathrm{AK}}
={}&
\left\{\{a,b\}\cup R:R\in\binom W{d-1}\right\}
\cup
\left\{\{a\}\cup S:S\in\cA_0\cup\cC_0\right\} \cup
\left\{\{b\}\cup S:S\in\cB_0\right\}.
\end{aligned}
\]
This is the lift of the covering pair \((\cA_0\cup\cC_0,\cB_0)\), whose overlap is \(\cC_0\). A direct count gives
\(|\mathcal L_{\mathrm{AK}}|=\binom{n-1}{d}+\binom{n-4}{d-2}\).

Let \(\Omega=\{\alpha,\beta,\gamma\}\).
The local profile family of \(\mathcal L_{\mathrm{AK}}\) over
\(\{a,b\}\cup \Omega\) is the following, where \(abT\) denotes
\(\{a,b\}\cup T\), and similarly for the other concatenations:
\begin{equation}\label{eq:R_AK}
\begin{aligned}
\cR_{\mathrm{AK}}
={}&
\{abT:T\subseteq \Omega\}  \\
&\cup
\{aT:T\subseteq \Omega,\ \alpha\notin T,\ \beta\in T\}
\cup
\{aT:T=\emptyset\}
\cup
\{aT:T\subseteq \Omega,\ \{\alpha,\beta\}\subseteq T\} \\
&\cup
\{bT:T\subseteq \Omega,\ \alpha\in T\}
\cup
\{bT:T\subseteq \Omega,\ \gamma\in T,\ \alpha,\beta\notin T\}.
\end{aligned}
\end{equation}
Moreover,
\[
        \mathcal L_{\mathrm{AK}}
        =
        \bigcup_{P\in\cR_{\mathrm{AK}}}
        \left\{P\cup A:A\in\binom V{d+1-|P|}\right\}.
\]
The proof of Theorem \ref{thm:main} identifies \(\mathcal L_{\mathrm{AK}}\) as the fixed part of a lifted covering pair before the \((\cC_1,\cD)\)-transfer.  The overlap in the corresponding covering pair \((\cA_0\cup\cC_0,\cB_0)\) is not terminal: through the \((\cC_1,\cD)\)-transfer, the local trace obstruction changes and the overlap enlarges from \(\cC_0\) to \(\cC_0\cup\cC_1\).  The new layer \(\cC_1\) is precisely the recursive port filled by the lower-dimensional family \(\cH\). Equivalently, the \((\cC_1,\cD)\)-transfer acts as a local switch that moves the subfamily \(\cD\) from the \(a\)-side to the \(b\)-side and creates the new recursive interface \(\cC_1\).

We record a further augmentation of the recursive
lifting in the language of local switches. The trace verification below shows that, after a higher-order count-preserving local switch, two additional recursive positions can be filled by further lower-dimensional families.

Write \(C=\{a,b,\alpha,\beta,\gamma,x\}\).  Each old profile \(P\in\cR_{\mathrm{AK}}\) splits into the two \(C\)-profiles \(P\) and \(xP\), according as the residual part avoids or contains \(x\).  Therefore the separated profile family is
\begin{equation}\label{eq:R_0}
        \cR_0=\cR_{\mathrm{AK}}\cup\{xP:P\in\cR_{\mathrm{AK}}\}.
\end{equation}
For a formal profile \(R\subseteq C\), write
\[ \operatorname{Lay}(R)=\left\{R\cup A:A\in\binom U{d+1-|R|}\right\},\]
with the convention that this layer is empty if
\(d+1-|R|<0\) or \(d+1-|R|>|U|\).  Then
\[
        \mathcal L_{\mathrm{AK}}
        =
        \bigcup_{R\in\cR_0}\operatorname{Lay}(R).
\]
 
\begin{prop}\label{prop:augmented-overlap}
With \(U,x,V,W,C\) as above, suppose that
\[
        \cH\subseteq\binom{\{x\}\cup U}{d-2},
        \qquad
        \VC(\cH)\le d-3, \quad\text{ and } \quad\cK_1,\cK_2\subseteq\binom U{d-3},
        \qquad
        \VC(\cK_i)\le d-4.
\]
Then there is a family $\cF\subseteq\binom{C\cup U}{d+1}$ with $\VC(\cF)\le d$ and
\[
        |\cF|=
        \binom{n-1}{d}+\binom{n-4}{d-2}+|\cH|+|\cK_1|+|\cK_2|.
\]
\end{prop}
\begin{proof}
By the setup above, \(\cR_0\) is the separated profile family of \(\mathcal L_{\mathrm{AK}}\). We perform the local switch
\(\{a\alpha\beta, b\alpha\beta, ab\alpha\beta\} \longrightarrow\{b\beta\gamma,\alpha\beta\gamma, x\alpha\beta\gamma\}\) on \(\cR_0\).
That is, define
\[
        \cR=
        \left(\cR_0
        \setminus\{a\alpha\beta,\ b\alpha\beta,\ ab\alpha\beta\}\right)
        \cup
        \{b\beta\gamma,\ \alpha\beta\gamma,\ x\alpha\beta\gamma\}.
\]
The switch replaces two profiles of size \(3\) and one profile of
size \(4\) by two profiles of size \(3\) and one profile of size \(4\).
Thus it preserves the multiset of profile sizes.
Let \(\cF_0=\bigcup_{R\in\cR}\operatorname{Lay}(R)\). 
Since
\(|\operatorname{Lay}(R)|\) depends only on \(|R|\), we have
\[
        |\cF_0|
        =
        \sum_{R\in\cR}|\operatorname{Lay}(R)|
        =
        \sum_{R\in\cR_0}|\operatorname{Lay}(R)|
        =
        |\mathcal L_{\mathrm{AK}}|
        =
        \binom{n-1}{d}+\binom{n-4}{d-2}.
\]

We next insert recursive families along four distinguished profiles:
\(\{a\alpha\gamma,ax\alpha\gamma,ab\alpha\beta,bx\beta\gamma\}\).
The first two are the two local appearances of the original \(\cH\)-insertion, according as a member of \(\cH\) avoids or contains \(x\).  The last two will carry the new lower-dimensional families \(\cK_1\) and \(\cK_2\).  The trace verification below will show that these insertions are legitimate.
Let \(\cF_H=\{a\alpha\gamma\cup H:H\in\cH\},\text{ }\cF_1=\{ab\alpha\beta\cup K:K\in\cK_1\},\text{ }\cF_2=\{bx\beta\gamma\cup K:K\in\cK_2\}\),
and define
\[
   \cF=\cF_0\cup\cF_H\cup\cF_1\cup\cF_2.
\]
We first count \(\cF\).
The four parts of $\cF$ are pairwise disjoint, since they are indexed by
distinct \(C\)-profiles and none of the recursive profiles belongs to
\(\cR\).
Thus
\[
        |\cF|=
        \binom{n-1}{d}+\binom{n-4}{d-2}
        +|\cH|+|\cK_1|+|\cK_2|.
\]

\begin{claim}\label{claim:augmented-trace}
Every member of \(\cF\) has a missing trace on itself.
\end{claim}

\begin{poc}
Put 
\begin{equation}\label{eq:P}
    \cP=\cR\cup\{a\alpha\gamma,\ ax\alpha\gamma,\ ab\alpha\beta,\ bx\beta\gamma\}.
\end{equation} 
This is the set of possible \(C\)-profiles of members of \(\cF\).  Also put
\(\mathcal E=\{b\beta\gamma,\ bx\alpha\beta,\ x\alpha\beta\gamma,\ ax\alpha\beta\}\subset\cR\).
We first consider a member \(F=R\cup A\in\cF_0\), where \(R\in\cR\) and
\(A\subseteq U\).  
\begin{itemize}
\item Suppose \(R\in\cR\setminus\mathcal E\).  Write
\(\cP_{\le s}=\{Q\in\cP:|Q|\le s\}\).
We choose a proper subset \(\sigma(R)\subsetneq R\) such that
\begin{equation}\label{eq:sigma-certificate}
        \{Q\in\cP_{\le |R|}:Q\cap R=\sigma(R)\}=\emptyset .
\end{equation}
Then \(A\cup\sigma(R)\) is missing from \(\Tr_{\cF}(F)\).  Indeed, if
\(G=Q\cup B\in\cF\) realized this trace, with \(Q=G\cap C\) and
\(B=G\cap U\), then
\(G\cap F=(Q\cap R)\cup(B\cap A)=A\cup\sigma(R)\).
Thus \(A\subseteq B\) and \(Q\cap R=\sigma(R)\).  Since
\(|Q|=d+1-|B|\le d+1-|A|=|R|\),
we have \(Q\in\cP_{\le |R|}\), contradicting \eqref{eq:sigma-certificate}.

It remains only to specify the certificates \(\sigma(R)\).  The non-empty certificates are listed in the following table, and for all profiles not listed
below we take \(\sigma(R)=\emptyset\). A direct check in the finite profile family \(\cP\) defined in
\eqref{eq:P} shows that, for every \(R\in\cR\setminus\mathcal E\), the
chosen \(\sigma(R)\) satisfies \eqref{eq:sigma-certificate}.
\[
\begin{array}{c|c@{\qquad}c|c@{\qquad}c|c}
R & \sigma(R) & R & \sigma(R) & R & \sigma(R)\\
\hline
ax & x
&
a\beta & \beta
&
b\alpha & \alpha
\\
b\gamma & \gamma
&
abx & x
&
ax\beta & x\beta
\\
a\beta\gamma & \beta
&
bx\alpha & x\alpha
&
bx\gamma & x\gamma
\\
b\alpha\gamma & \alpha
&
\alpha\beta\gamma & \alpha\beta
&
ax\beta\gamma & \beta
\\
bx\alpha\gamma & \alpha
&
b\alpha\beta\gamma & \alpha
&
bx\alpha\beta\gamma & \alpha
\end{array}
\]

\item For  \(R \in\mathcal E\).  In the table
below, choose \(Y\subseteq S\) with \(Y\notin\Tr_{\mathcal G}(S)\); such a
\(Y\) exists from the VC assumptions on \(\cH,\cK_1,\cK_2\).  The last column
gives a trace missing from \(F=R\cup A\):
\[
\begin{array}{c|c|c|c}
R & \mathcal G & S & \text{missing trace}\\
\hline
b\beta\gamma & \cH & A & \gamma\cup Y\\
bx\alpha\beta & \cH & \{x\}\cup A & \alpha\cup Y\\
x\alpha\beta\gamma & \cK_1 & A & \alpha\beta\cup Y\\
ax\alpha\beta & \cK_2 & A & x\beta\cup Y.
\end{array}
\]
For the four rows, the only profiles \(Q\in\cP\) whose intersection with
\(R\) equals the \(C\)-part of the proposed missing trace are as follows: the only profiles in \(\cP\)
meeting \(b\beta\gamma\) exactly in \(\gamma\) are
\(a\alpha\gamma\) and \(ax\alpha\gamma\), both coming from \(\cF_H\).  For
\(bx\alpha\beta\), the profiles meeting it in \(\alpha\) or \(x\alpha\) are
respectively \(a\alpha\gamma\) and \(ax\alpha\gamma\).  For
\(x\alpha\beta\gamma\), the only profile meeting it in \(\alpha\beta\) is
\(ab\alpha\beta\).  For \(ax\alpha\beta\), the only profile meeting it in
\(x\beta\) is \(bx\beta\gamma\).  Therefore, in each row, a realization would
force a member of the corresponding family \(\mathcal G\) with trace \(Y\)
on \(S\), impossible by the choice of \(Y\).
\end{itemize}

It remains to consider the recursive parts.  
A direct check in \(\cP\) shows that every \(Q\in\cP\) meets each of
\(a\alpha\gamma\), \(ax\alpha\gamma\), and \(ab\alpha\beta\).  Hence the
empty trace is missing on every member of \(\cF_H\) and every member of
\(\cF_1\). Thus we further assume \(F=bx\beta\gamma\cup K\in\cF_2\).  Put
\(S=\{x\}\cup K\), and choose \(Y\subseteq S\) with
\(Y\notin\Tr_{\cH}(S)\).  We claim that \(\gamma\cup Y\) is missing.  If
\(x\notin Y\), then the only possible realizing profile is
\(a\alpha\gamma\); if \(x\in Y\), then the only possible realizing profile is
\(ax\alpha\gamma\).  In either case the realizer would have to come from
\(\cF_H\), forcing a member of \(\cH\) with trace \(Y\) on \(S\), a
contradiction.  Thus every member of \(\cF\) has a missing trace on itself.
\end{poc}

By Claim \ref{claim:augmented-trace} and Lemma \ref{lem:uniform}, we have
\(\VC(\cF)\le d\).  Thus the four distinguished profiles above are indeed valid recursive insertion positions.  Together with the count above, this proves the proposition.
\end{proof}

\begin{cor}\label{cor:augmentation}
For every $d\ge4$ and $n\ge d+3$,
\begin{equation}\label{eq:augmented}
        M_d(n)\ge
        \binom{n-1}{d}+\binom{n-4}{d-2}
        +M_{d-3}(n-5)+2M_{d-4}(n-6).
\end{equation}
\end{cor}

\begin{proof}
Choose $\cH,\cK_1,\cK_2$ in Proposition \ref{prop:augmented-overlap} with
\(|\cH|=M_{d-3}(n-5),\text{ }|\cK_1|=|\cK_2|=M_{d-4}(n-6)\).
Then Proposition \ref{prop:augmented-overlap} gives a $(d+1)$-uniform family with VC-dimension at most $d$ and size equal to the right-hand side of \eqref{eq:augmented}.
\end{proof}

\section*{Acknowledgements}
The authors thank Tuan Tran and Zixiang Xu for making their independent and nearly simultaneous work available on arXiv.  Following the appearance of the earlier preprint~\cite{ZhaoGe2026v1}, their parallel construction partly motivates the further recursive augmentation presented in Section~\ref{sec:concluding}.

\bibliographystyle{abbrv}
\bibliography{references}
\end{document}